\documentclass[journal,9 pt]{IEEEtran}
\IEEEoverridecommandlockouts                            
\usepackage{graphics} % for pdf, bitmapped graphics files
\usepackage{epsfig} % for postscript graphics files
\usepackage{mathptmx} % assumes new font selection scheme installed
\usepackage{times} % assumes new font selection scheme installed
\usepackage{amsmath} % assumes amsmath package installed
\usepackage{amssymb}  % assumes amsmath package installed
\usepackage{epstopdf}
\usepackage{cite}
\usepackage{mathtools}
\usepackage[caption=false, font=footnotesize]{subfig}
\usepackage{algorithm,algorithmic}
\usepackage{color}
\usepackage{amsthm}
\usepackage{booktabs}
\usepackage{fancyhdr}
\newtheorem{thm}{Theorem}
\newtheorem{lmm}{Lemma}
\newtheorem{assumption}{Assumption}
\newtheorem{rmk}{Remark}
\newtheorem{prp}{Proposition}
\newtheorem{corollary}{Corollary}
\newcommand\reals[0]{\mathbb{R}}

\newcommand\eps[0]{\varepsilon}

\title{%\LARGE \bf
A Regularized and Smoothed Fischer-Burmeister Method for Quadratic Programming with Applications to Model Predictive Control
}

\author{Dominic Liao-McPherson$^{1}$, Mike Huang$^{2}$, and Ilya Kolmanovsky$^{1}$% <-this % stops a space
\thanks{$^{1}$D. Liao-McPherson and I. Kolmanovsky are with the University of Michigan, Ann Arbor. Email:\{dliaomcp,ilya\}@umich.edu}%
\thanks{$^{2}$M. Huang is with Toyota Motors North America R\&D, Ann Arbor, Michigan. Email:\{mike.huang\}@toyota.com}%
}
\begin{document}

% The paper headers
\markboth{Submitted to IEEE Transactions on Automatic Control}%
{Shell \MakeLowercase{\textit{et al.}}: Bare Demo of IEEEtran.cls for IEEE Journals}

% make the title area
\maketitle

\begin{abstract}
This paper considers solving convex quadratic programs (QPs) in a real-time setting using a regularized and smoothed Fischer-Burmeister method (FBRS). The Fischer-Burmeister function is used to map the optimality conditions of the quadratic program to a nonlinear system of equations which is solved using Newton's method. Regularization and smoothing are applied to improve the practical performance of the algorithm and a merit function is used to globalize convergence. FBRS is simple to code, easy to warmstart, robust to early termination, and has attractive theoretical properties, making it appealing for real-time and embedded applications. Numerical experiments using several predictive control examples show that the proposed method is competitive with other state of the art solvers.
\end{abstract}

\IEEEpeerreviewmaketitle

\section{Introduction}
% \IEEEPARstart{T}{his} demo file is intended to serve as a ``starter file''
% for IEEE journal papers produced under \LaTeX\ using
% IEEEtran.cls version 1.8b and later.

Real-time optimization has the potential to dramatically improve the capabilities of many systems. A key class of optimization problems in real-time applications involves convex quadratic programs (QPs); many practical problems in control, signal processing, machine learning and other domains can be posed as convex QPs \cite{boyd2004convex}. In recent years significant progress has been made developing fast, reliable algorithms for solving both QPs and more general optimization problems online. However, many applications, especially fast systems with limited computing power, remain challenging.

An important instance of a real-time, embedded optimization problem is the one in model predictive control (MPC) \cite{grune2011nonlinear} \cite{rawlings2009model} \cite{goodwin2006constrained}, where an optimal control problem over a receding horizon is solved during each sampling period. The optimal control problem for a discrete time linear-quadratic MPC formulation can be expressed as a convex QP. Furthermore, convex QPs form the basis for many algorithms used in nonlinear model predictive control (NMPC) such as sequential quadratic programming (SQP) \cite{boggs1995sequential}, and the real-time iteration scheme\cite{diehl2002real} which solves just one QP per timestep.

Interest in embedded optimization has motivated extensive research into fast, reliable solvers tailored for embedded systems. Algorithms and solvers specialized for MPC include the algorithms of Wang et al.\cite{wang2010fast} and FORCES \cite{domahidi2012efficient}, which are based on interior point methods, and qpOASES \cite{ferreau2008online} which is based on the active set method. Other algorithms, which are often used for MPC but which can solve more general convex QPs, include GPAD \cite{patrinos2014accelerated}, CVXGEN \cite{mattingley2012cvxgen}, NNLS \cite{bemporad2016quadratic}, piecewise smooth Newton's methods \cite{li1997new,patrinos2011global}, and PQP \cite{brand2011parallel}. In addition, solvers and algorithms for embedded second order cone programs, have begun to appear\cite{domahidi2013ecos} \cite{dueri2014automated}.

This paper considers the application of the Fischer-Burmeister (FB) function and a smoothing Newton's method to solving convex QPs. The necessary conditions for optimality are mapped to a system of non-smooth equations using the Fisher-Burmeister function. The equations are then smoothed and Newton's method is then applied to solve the resulting root finding problem. Regularization and a line search are added to control the numerical conditioning of the linear subproblems and enforce global convergence.

The regularized and smoothed FB algorithm, which we will refer to as FBRS (Fischer Burmeister Regularized Smoothed), has nice properties which make it attractive for embedded optimization. Firstly, FBRS displays global convergence and quadratic asymptotic convergence, properties it inherits from its nature as a damped generalized Newton's method. Secondly, it is simple to implement, a complete embeddable implementation is possible in under 100 lines of MATLAB code. Finally, it can be effectively warmstarted when solving sequences of related QPs, which is beneficial in many real-time optimization problems, including MPC.

Fischer-Burmeister (FB) functions in conjunction with both smoothing and semismooth Newton's method have been investigated in the past; a version of this algorithm, without smoothing or regularization, applied to general nonlinear programs was studied in \cite{facchinei1998regularity}. In addition, some smoothing methods for linear complementarity problems, which subsume box constrained convex QPs, have been proposed, see e.g., \cite{huang2004sub} and the references therein. However, this paper investigates its use for quadratic programming at a level of depth and detail not present in the previous literature. Furthermore, we consider its suitability for embedded use and perform numerical experiments demonstrating its applicability to predictive control. In addition, FBRS includes practical improvements such as regularization to handle ill-conditioned Jacobians.% and the use of the C-differential \cite{qi1996c} to ease the computation of generalized derivatives. 

% compare with other methods
FBRS also has some advantages when compared with other methods for solving convex QPs. Firstly, FBRS smooths the complementarity conditions in a manner similar to an interior point algorithm; however, unlike an interior point method, FBRS is locally quadratically convergent with no smoothing. This is the key property of FBRS which makes it attractive for solving sequences of related QPs. Secondly, FBRS has a faster convergence rate than first-order methods such as dual accelerated gradient projection (GPAD), the alternating direction method of multipliers (ADMM), or multiplicative update methods such as PQP, and does not require that the QP be strictly convex. Finally, in contrast with active set or primal barrier interior point methods, the initial guess need not be feasible.

% other work with FB and semi-smooth
% add references to smoothing Newton's methods
Fisher \cite{fischer1992special} used the eponymous function to map the Karush-Kuhn-Tucker (KKT) conditions for a nonlinear program to a nonlinear system of equations which is then solved with a non-smooth Newton's method based on Clarke's generalized Jacobian. Local convergence results are obtained but globalization is not considered. In \cite{chen2000penalized} a semismooth Newton's method which uses a penalized FB function to solve nonlinear complementarity problems is presented which uses a version of the C-differential and globalizes the algorithm using a line-search. The application of the FB function to diesel engine MPC was considered in \cite{huang2015nonlinear}. %Kanzow et al. \cite{kanzow1999qp} consider the solution of variational inequalities by the application of a FB function based non-smooth Newton's method to the KKT conditions of the variational inequality; global and local convergence results are obtained. 

Some key papers concerning generalized Newton's methods and their convergence are \cite{qi1993nonsmooth} and \cite{qi1997semismooth}. Reference \cite{martinez1995inexact} concerns the convergence of inexact generalized Newton's methods. Reference \cite{qi1999survey} is a useful survey on the topic.

% Smoothing Newton's methods are explored in \cite{chen1998global}, \cite{jiang1997smoothed}, and \cite{gabriel1997smoothing}. 
% Finally, 
% References \cite{martinez1995inexact} and \cite{facchinei1997nonsmooth} concern the convergence of inexact generalized Newton's methods.
%Smoothing Newton's methods are explored in \cite{chen1998global}, \cite{jiang1997smoothed}, and \cite{gabriel1997smoothing}. 

% notation
\textit{Notation}: Our notations are standard. $\reals^n$ denotes the set of $n$-dimensional real vectors. For a vector $x\in \reals^n$ $x_i$ denotes its $i$-th entry and the relations $\leq,\geq ,< ,>$ are understood component wise. For a matrix $A\in \reals^{m \times n}$ $A_i$ denotes the $i$-th row of the matrix; if $I$ is an index set $A_I$ denotes the concatenation of all $A_i,~i \in I$. Let $S_{+(+)}^n$ denote the set of $n \times n$ (strictly) symmetric positive definite matrices, $A\succ 0, A\succeq 0$ denote positive definiteness and semi-definiteness, respectively. The kernel of a map $T$ is denoted by $\ker{T}$ and $I$ is used to denote the identify matrix, the dimensions of which should be clear from context. Set unions, intersections, and subtractions are denoted by  $\cup,~\cap$, and $\setminus$ respectively. The cardinality of a set $S$ is denoted by $|S|$. A matrix or vector norm $||\cdot||$ will be taken to indicate the two norm unless otherwise indicated. Let $h:\reals^n \to \reals^m$ and $g(x):\reals^n \to \reals_{\geq 0}$. We write $h(x) = O(g(x))$ as $x \to \bar{x}$ if $\exists M > 0$ such that $||h(x)|| \leq M g(x)$ for all $x$ sufficiently close to $\bar{x}$. If $||h(x)|| \leq \eps g(x),~ \forall~\eps > 0$ holds for all $x$ sufficiently close to $\bar{x}$ then we say $h(x) = o(g(x))$.

% problem formulation
\section{Problem formulation} \label{ss:QP_formulation}
This paper considers solving convex QPs in $n$ variables with $q$ constraints of the form
\begin{subequations}\label{eq:QP}
\begin{gather} 
\underset{z}{\mathrm{min.}} \quad g(z) = \frac12 z^T H z + f^T z, \\
\mathrm{s.t} \quad Az \leq b,
\end{gather}
\end{subequations}
where $H \in S^{n}_+ $ is the Hessian matrix, $f\in \reals^{n}$, $z\in \reals^n$, $A\in\reals^{q \times n}$, and $b \in \reals^q$. For simplicity we consider the case where there are no equality constraints. The extension to equality and inequality constrained problems is straightforward, alternatively polyhedral inequality and equality constrained problems can always be converted into a purely inequality constrained problem provided the equality constraints are feasible, see e.g., \cite[Section 4.1.3]{boyd2004convex}.
The Lagrangian for this problem is
\begin{equation}  \label{eq:Lagrangian}
  L(z,v) = \frac12 z^T H z + f^T z + v^T(Az -b),
\end{equation}
and the KKT conditions for \eqref{eq:QP} are
\begin{subequations}\label{eq:KKT}
\begin{gather}
\nabla_z L = Hz + f + A^T \nu = 0,\\
\nu_i\cdot y_i = 0,~i = 1~...~q,\label{eq:comp1}\\
\nu \geq 0,~~ y\geq 0, \label{eq:comp2}
\end{gather}
\end{subequations}
where we let,
\begin{equation}
   y = b-Az,
\end{equation} 
denote the constraint residual. Since $H$ is positive semidefinite these conditions are necessary and sufficient for optimality under an appropriate constraint qualification. We are interested in the case where \eqref{eq:QP} has a unique primal-dual solution, so we impose some additional assumptions on the problem.

Let $x^* = (z^*,v^*)$ denote a point satisfying \eqref{eq:KKT}, referred to as a critical or KKT point, and let $I_a(z) = \{i\in 1~...~ q~|~ A_i z = b_i\}$ denote the set of active constraints at $z$. Recall that for a system of linear inequalities of the form $Az\leq b$, the linear independence constraint qualification (LICQ) is said to hold at a point $\bar{z}$ if 
\begin{equation}
  \text{rank}~A_{I_a(\bar{z})} = |I_a(\bar{z})|,
\end{equation}
and that if 
\begin{equation}
  u^T H u> 0,~\forall u \neq 0~\text{such that}~A_i u = 0,~\forall i \in I_a^+(z^*,v^*),
\end{equation} where $I_a^+(z,v) = \{i \in 1~...~ q~|~ A_iz = b_i,~v_i > 0\}$, then $x^*$ is said to satisfy the strong second order sufficient conditions (SSOSC). The following two assumptions are then sufficient for local primal-dual uniqueness.
\begin{assumption} \label{assump:SSOSC}
(\textbf{A1}) There exists a point $x^* = (z^*,v^*)$ that satisfies the strong second order sufficient conditions.
\end{assumption}
\begin{assumption} \label{assump:LICQ}
(\textbf{A2}) The linear independence constraint qualification (LICQ) holds at $x^*$.
\end{assumption}
Since \eqref{eq:QP} is convex, the SSOSC implies that $z^*$ is the unique global minimizer of \eqref{eq:QP}. The LICQ is used in place of Slater's condition, because the LICQ implies that the dual variable $v^*$ associated with $z^*$ is unique and thus the primal-dual solution is isolated; this simplifies the convergence analysis. In the degenerate case where either the primal or dual solution is not isolated an algorithm in the same vein as the stabilized Josephy-Newton method may be applicable, see e.g., \cite[Chapter 7]{izmailov2014newton}.

% If \eqref{eq:QP} is infeasible then a extended problem can be generated using primal-dual embedding, see e.g., \cite{de1997initialization}. However, in real-time applications feasibility is often ensured by construction, using e.g., slack variables. 

Since $H$ is assumed to be positive semidefinite rather than strictly positive definite we add an additional assumption to ensure that intermediate iterations are well defined.
\begin{assumption} \label{assump:ker_intersection}
(\textbf{A3}) $\ker{H}~\cap~\ker{A} = \{0\}$
\end{assumption}
Intuitively, this condition requires all directions along which the cost function has no curvature to be constrained. For a QP this condition can be checked numerically. It is often desirable for (A\ref{assump:ker_intersection}) to be satisfied by construction in an embedded context; typically using regularization. Note that $H$ is only assumed to be positive semidefinite, rather than strictly positive definite; it allows \eqref{eq:QP} to capture a wider range of QP problems. If $H$ is strictly positive definite then the SSOSC and (A\ref{assump:ker_intersection}) are satisfied automatically.

%\cite{rademacher1919partielle} 

\section{Some concepts from non-smooth analysis} \label{ss:concepts}
In this section we review some key concepts from non-smooth analysis which are required to motivate and analyze FBRS. We begin with generalized differentiation. Suppose a function $G : \reals^N \mapsto \reals^M$ is locally Lipschitz on a set $U \subseteq \reals^N$, i.e. $\exists~L>0~~s.t~~ ||G(x+\xi) - G(x)|| \leq L||\xi||,~\forall \xi \in U$. Then Rademacher's theorem \cite{rademacher1919partielle} states that $G$ is differentiable almost everywhere. Letting $D_G$ denote the set of points where $G$ is differentiable, the B-differential is defined as
\begin{equation}
  \partial_B G(x) = \{J\in \reals^{M\times N}|~ \exists \{x^k\} \subset D_G : \{x^k\} \rightarrow x,~ \{\nabla G(x_k)\} \rightarrow J\},
\end{equation}
and Clarke's generalized Jacobian, $\partial G(x) = \text{convh}~\partial_B G(x),$ can be defined as the convex hull of the B-differential \cite{clarke1990optimization}. The generalized Jacobian is a set of matrices, wherever $G$ is differentiable $\nabla G(x) \in \partial G(x)$ and $\partial G(x) = \{\nabla G(x)\}$ wherever $G$ is continuously differentiable \cite{izmailov2014newton}.

% semi-smoothness
A mapping $G:\reals^N \mapsto \reals^M$ is said to be semismooth at $x \in \reals^N$ if $G$ is locally Lipschitz at $x$, directionally differentiable in every direction and the estimate
\begin{equation}
  \underset{J \in \partial G(x+\xi)}{\text{sup}} ||G(x+\xi) - G(x) - J\xi|| = o(||\xi||),
\end{equation}
holds. If the right hand side is replaced by the stronger bound $O(||\xi||^2)$ then $G$ is said to be strongly semismooth at $x$ \cite{izmailov2014newton}. The concept of semismoothness plays a key role in the analysis of non-smooth Newton's methods \cite{qi1993nonsmooth}.

% A mapping $G:\reals^N \mapsto \reals^N$ is said to be CD regular \cite{qi1997semismooth} at $x$ if all $J \in \partial G(x)$ are non-singular. This property is closely related to the existence and Lipschitz continuity of a local inverse mapping $G^{-1}$ and reduces to the non-singularity of the Jacobian if $G$ is a continuously differentiable function.\\

% c-differential
In this paper the C-differential of a mapping, denoted $\partial_C G$, is defined as
\begin{equation}
  \partial_C G = \partial G_1 \times \partial G_2 \times ...~\partial G_{N},
\end{equation}
where $\partial G_i$ is the transpose of the so-called generalized gradient of $G_i$, which is simply the generalized Jacobian \cite{clarke1990optimization} of the component mapping $G_i: \reals^N \mapsto \reals$. Note that $\partial G_i$ is a row vector, thus $\partial_C G$ is a set of matrices whose rows are the transposed generalized gradients of the component functions $G_i$. This form of the C-differential was introduced in \cite{chen1998global} and \cite{qi1996c}, it possesses many of the properties of the generalized Jacobian but can be easier to compute and characterize.

\section{The Algorithm and its Properties}
FBRS (Fischer-Burmeister Regularized and Smoothed) is a complimentary mapping algorithm based on Newton's method which uses smoothing to ensure robust globalization.  FBRS, and complementarity mapping methods in general, function by bijectively mapping the complementarity conditions \eqref{eq:comp1} and \eqref{eq:comp2} to a system of equations using what is known as a nonlinear complementarity (NCP) function\cite{sun1999ncp}. We will use the generalized (smoothed) Fischer-Burmeister function
\begin{equation}
  \phi_\eps(a,b) = a + b - \sqrt{a^2 + b^2 + \eps^2},
\end{equation}
which has the property that
\begin{equation}
  \phi_\eps(a,b) = 0 \Leftrightarrow a\geq 0,~b\geq 0,~\sqrt{2} ab = \eps,
\end{equation}
and that
\begin{equation} \label{eq:phi_eps_bound}
  ||\phi_\eps(a,b) - \phi(a,b)||\leq \eps, \quad \phi(a,b) = \phi_0(a,b).
\end{equation}

%Reference \cite{chen2008family} has derivations for some properties of the FB function.\\

Applying the FB function to the complementarity conditions\footnote{Throughout this paper the application of the FB function to vectors i.e., $\phi_\eps(x,y)$ for $x,y \in \reals^n$, will be understood to be elementwise.} in \eqref{eq:KKT} yields the following nonlinear mapping,
\begin{equation} \label{eq:F}
F_\eps(x) = \begin{bmatrix}
\nabla_zL(z,v)\\
\phi_\eps(v,y)\\
\end{bmatrix}, \quad F(x) = F_0(x)
\end{equation}
where $x = (z,v)$ denotes the primal-dual pair. The following two results are direct consequences of the properties of the FB function and relate the smoothed and original problems.

\begin{corollary} \label{corr:root_unique}
Under the assumptions in Section~\ref{ss:QP_formulation} the mapping $F_0(x) = F(x) = 0$ if and only if $x = x^*$, is the solution to \eqref{eq:QP}. Further, this root is unique.
\end{corollary}
\begin{proof}
Since the problem is convex, the SSOSC implies uniqueness of $z^*$ and the LICQ implies uniqueness of $v^*$. Roots of $F(x)$ coincide with points which satisfy the KKT conditions \cite{fischer1992special}. Thus $x^*$ uniquely satisfies $F(x^*) = 0$.
\end{proof}
\begin{corollary} \label{corr:smoothing_bound}
For all $x \in \reals^{n+q}$ and $\eps \geq 0$ we have that $||F_\eps - F|| \leq \sqrt{q} \eps$ and  
\begin{equation}
  ||F(x)|| \leq ||F_\eps(x)|| + \sqrt{q}\eps.
\end{equation}
\end{corollary}
\begin{proof}
Direct computation yields,
\begin{equation}
||F(x)_\eps - F(x)|| = \left |\left|\begin{bmatrix}
  0 \\ \phi_\eps(y,v) - \phi(y,v)
\end{bmatrix} \right |\right|,
\end{equation}
applying \eqref{eq:phi_eps_bound}, the properties of the two norm, and the reverse triangle inequality then yields the result. 
\end{proof}
\medskip
\tabularnewline

The general FBRS algorithm approximately solves a sequence of sub-problems $F_{\eps_k}(x_k) = 0$ for a decreasing sequence, $\eps_k \to 0$. Each sub-problem is solved using Newton's method
\begin{equation} \label{eq:base_newton}
  x_{k+1} = x_k - t_k V_k^{-1}F_{\eps_k}(x_k),
\end{equation}
where $V_k = \nabla_x F_{\eps_k}(x_k)$ and $t_k \in (0,1]$ is a steplength chosen by a linesearch to enforce convergence far from a solution. 

The iteration matrix or Jacobian $V_k$ is always non-singular for any $\eps \geq 0$ (see Section~\ref{ss:directions}, Theorem~\ref{thm:V_nonsingularity}) but can become ill-conditioned, so a regularization term is added. Defining $R(x,\delta) = \left[0^T~~\delta (v^T+y^T) \right]^T$, where $\delta \geq 0$ is the regularization strength, we can replace $V_k$ with $K_k = V_k + \nabla_x R(x_k,\delta_k)$ in \eqref{eq:base_newton}, leading to the smoothed and regularized Newton iteration
\begin{equation} \label{eq:newton}
  x_{k+1} = x_k - t_k K_k^{-1}F_{\eps_k}(x_k),
\end{equation} 
which forms the core of the FBRS method. Expressions for computing $K$ and an analysis of its properties are presented in Section~\ref{ss:directions}.

\begin{rmk} \label{rmk:FBRSss}
A semismooth version of FBRS results if in \eqref{eq:base_newton} $V_k$ is redefined as $V_k \in \partial_C F_{\eps} (x_k)$. This version allows $\eps = 0$ and reduces to the smoothed version for $\eps > 0$ (since then $\partial_C F_{\eps} = \{\nabla_x F_\eps\}$). A semismooth FB method without regularization has been previously proposed in the literature, see e.g., \cite{facchinei1998regularity} or \cite[Section 5.1.2]{izmailov2014newton} and is globally convergent for convex QPs. However, we have observed that both smoothing and regularization improve the numerical performance of the algorithm and avoid the need to compute generalized derivatives. 
%Further, the application of the smoothed algorithm to more general convex programming problems appears more straightforward than the semismooth variant as the conditions required for global convergence of the semismooth version \cite[Theorem 4.3]{facchinei1998regularity} are difficult to satisfy when the constraints are nonlinear.
The semismooth variant of FBRS can be shown (see Section~\ref{ss:convergence}, Theorem~\ref{thrm:asym_conv}) to be locally quadratically convergent when $\eps \geq 0$. This property distinguishes NCP function based smoothing methods from interior points methods, where the barrier strength can only approach zero in the limit. This property means that the FBRS subproblems do not become ill-conditioned even when $\eps$ is very small, facilitating warmstarting.
\end{rmk}

FBRS is summarized in Algorithm~\ref{algo:FBRS} and is simply Newton's method globalized using a linesearch and homotopy. The merit function used to globalize each subproblem in the algorithm is defined as
\begin{equation}
  \theta_\eps(x) = \frac12 ||F_\eps(x)||_2^2,
\end{equation}
and the parameter $\sigma \in (0,0.5)$ encodes how much reduction we require in the merit function. The desired solution tolerance is denoted $\tau$, and $\beta \in (0,1)$ controls reduction in the backtracking linesearch. More sophisticated algorithms for computing $t_k$, e.g., polynomial interpolation, can be used in place of the backtracking linesearch; however we found this backtracking to be effective in practice. Typical values for the fixed parameters are $\sigma \approx 10^{-4}$ and $\beta \approx 0.7$. 

\begin{algorithm}[H]
\caption{FBRS}
\begin{algorithmic}[1] \label{algo:FBRS}
 \renewcommand{\algorithmicrequire}{\textbf{Input:}}
 \renewcommand{\algorithmicensure}{\textbf{Output:}}
 \REQUIRE $H$, $A$, $f$, $b$, $x_0$, $\sigma$, $\beta$, $\tau$, $\delta_0$, \texttt{max\_iters}
 \ENSURE  $x$
  \STATE $x \gets x_0$, $\eps \gets \frac{\tau}{2\sqrt{q}}$, $\delta \gets \delta_0$
  \FOR{\texttt{k = 0 to max\_iters-1}}
    \STATE $\delta \gets \text{min}(\delta,||F_\eps(x)||)$ \label{step:reg_update}
    \IF{$||F_0(x)|| \leq \tau$}
      \STATE \texttt{break;}
    \ENDIF
    \STATE Solve $K(x,\eps,\delta)\Delta x = -F_\eps(x)$ for $\Delta x$ \label{step:solve_step}
    % \IF{$\nabla \theta(x)^T \Delta x \geq 0$}
    %   \STATE $\Delta x \gets -\nabla \theta(x)$
    % \ENDIF
    \STATE $t \gets 1$
    \WHILE{$\theta_\eps(x+ t \Delta x) \geq (1- 2t\sigma) \theta_\eps(x)$}
      \STATE $t \gets \beta t$ 
    \ENDWHILE
    \STATE $x \gets x+ t \Delta x$
  \ENDFOR
  \RETURN $x$
 \end{algorithmic}
 \end{algorithm}

For any fixed $\eps > 0$ FBRS, under the assumptions in Section \ref{ss:QP_formulation}, exhibits global linear convergence (see Section~\ref{ss:convergence}, Theorem~\ref{thm:global_convergence}) and local quadratic convergence (see Section~\ref{ss:convergence}, Theorems~\ref{thrm:asym_conv} and \ref{thm:q_recovery}) to the unique point satisfying $F_\eps(x) = 0$. For simplicity, the assumption that $\delta_k$ is always ``small enough'' that global convergence is not impeded is implicit in Algorithm~\ref{algo:FBRS}. If $\delta_k$ is too big at any iteration the linesearch may fail. In this case $\delta_k$ can be reduced and a new step computed or a gradient descent step on the merit function can be taken.

\begin{rmk}
For embedded applications we often only require moderate precision solutions, e.g., $||F|| \approx 10^{-6}$ to $10^{-8}$ in double precision. In these situations we found fixing $\eps$ at a small value is sufficient and simplifies warmstarting by removing the need to reinitialize $\eps$. In this paper we simply take $\eps_k = \frac{\tau}{2\sqrt{q}}$. More sophisticated strategies for updating $\eps$ may be helpful for improving numerical performance. Inspiration could potentially be drawn from the varied barrier update rules used in interior point methods \cite{nocedal2006numerical}. If high precision solutions are required the semismooth ($\eps = 0$) variant of FBRS can be used and could be warmstarted using the smoothed algorithm. 
\end{rmk}

\begin{rmk}
In practice since $K_k$ is guaranteed to be nonsingular even for $\delta = 0$ the regularization parameter exists solely to handle numerical ill-conditioning and a small fixed value can be used throughout. For example $\delta_k = \delta_0 = 10^{-8}$ was used throughout this paper.
\end{rmk}

\section{The Iteration Matrix and Computation of Step Directions} \label{ss:directions} This section presents a more detailed analysis of the properties of the iteration matrix $K$, whose factorization is the main computational burden of FBRS.
The Newton step system is
\begin{equation}
  \begin{bmatrix} \label{eq:newton_system}
    H & A^T\\
    -CA& D
  \end{bmatrix} \begin{bmatrix}
    \Delta z\\ \Delta \nu
  \end{bmatrix}  = \begin{bmatrix}
    - \nabla_z L(z,v)\\
-\phi_\eps(\nu,y)
  \end{bmatrix} = \begin{bmatrix}
    r_s\\ r_c
  \end{bmatrix},
\end{equation}
the matrices $C = diag(\gamma)$ and $D = diag(\mu)$ are diagonal matrices constructed by regularizing the gradient of the smoothed FB function. If $\eps > 0$ the diagonal elements of $C$ and $D$ are,
\begin{equation} \label{eq:CD_smooth}
  \gamma_i = 1-\frac{y_i}{\sqrt{y_i^2 + v_i^2 + \eps^2}} + \delta,~ \mu_i = 1-\frac{v_i}{\sqrt{y_i^2 + v_i^2 + \eps^2}} + \delta,
\end{equation}
further, it is evident that $C \succ 0$ and $D \succ 0$ if $\eps >0$.
\begin{rmk}
The generalized gradient of the FB function is well known, see e.g., \cite{jiang1997smoothed}; when using the semismooth version of FBRS $K$ can be computed by noting that $\gamma_i$ and $\mu_i$ are multivalued only if $\eps = 0$ and $(v_i,y_i) = 0$; then any choice of $\gamma_i$ and $\mu_i$ that statisfies
\begin{equation}
  \gamma_i = 1-\alpha+\delta,~ \mu_i = 1-\beta+\delta ~~\text{s.t} \quad \alpha^2 + \beta^2 \leq 1,
\end{equation}
implies that $K \in \partial_C F_\eps(x,\eps) + \nabla_x R(x,\delta)$.
\end{rmk}
The Jacobian matrix $V = \nabla_x F_\eps$ can be shown to be nonsingular.
\begin{thm} \label{thm:V_nonsingularity}
Let Assumption~\ref{assump:ker_intersection} hold and pick $\eps > 0$. Then $V = \nabla_x F_\eps(x)$ is non-singular $\forall x \in \reals^{n+q}$.
\end{thm}
\begin{proof}
Recall that since $V$ is smoothed but not regularized ($\delta = 0$) if $\eps > 0$ then $D \succ 0$, and thus $V$ can be factored blockwise as
\begin{multline} \label{eq:factor}
  V  =
  \begin{bmatrix}
    I & A^T D^{-1} \\
    0 & I
  \end{bmatrix} 
  \begin{bmatrix}
    H + A^T D^{-1}C A & 0 \\
    0 & D
  \end{bmatrix}
  \begin{bmatrix} 
    I & 0 \\
    -D^{-1} C A & I
  \end{bmatrix}.
\end{multline}
The upper and lower triangular factors are necessarily invertible, as is $D$, thus only $T = H + A^T D^{-1}C A$ must to be analyzed; we will show that $T \succ 0$. Letting $L = \sqrt{D^{-1}C}A$, it's clear that $T = H + L^TL \succeq 0$. To show positive definiteness assume there exists
\begin{equation}
   z \neq 0 \text{ such that } z^TT z = z^THz + (Lz)^T Lz = 0,
 \end{equation}
since $H \succeq 0$ and $L^T L \succeq 0$, $z^T T z = 0$ implies that $Hz = 0$ and $Lz = \sqrt{D^{-1}C}A z = 0$. Since $D^{-1}C$ is diagonal and positive definite this contradicts Assumption~\ref{assump:ker_intersection} ($\ker H \cap \ker A = \{0\}$). As a result $H + A^T D^{-1}CA$ must be positive definite and each factor in \eqref{eq:factor} is invertible.
\end{proof}

\begin{rmk} \label{rmk:K_compute}
It turns out that for the special case of convex QPs all elements of $\partial_C F_\eps$ are non-singular when $\eps \geq 0$ \cite[Theorem 4.4]{facchinei1998regularity}. The proof is involved and we have observed that the smoothed version works well in practice; as a result we have elected to present a self-contained proof of Theorem~\ref{thm:V_nonsingularity} for the smoothed version.
\end{rmk}

\begin{corollary} \label{corr:K_nonsingular}
Let assumption A\ref{assump:ker_intersection} hold, $\delta > 0$, and $\eps > 0$, then all $K = V + \nabla R(x,\delta)$ are non-singular.
\end{corollary}

\begin{proof}
Recall that if $\delta > 0$ then $C$ and $D$ are both positive definite for all $\eps > 0$. Then repeat the proof of Theorem~\ref{thm:V_nonsingularity}.
\end{proof}

The main computational burden of FBRS is solving the linear system, $K\Delta x = -F_\eps$, see Algorithm~\ref{algo:FBRS}. If there is no exploitable structure then an LU decomposition is a practical choice. If $K$ is sparse or matrix-vector products can be computed quickly then an iterative method may be appropriate. Alternatively, using the same block-LU decomposition as in the proof of Theorem~\ref{thm:V_nonsingularity} a condensed decomposition can be derived
\begin{subequations}\label{eq:Kred}
\begin{gather}
(H + A^T CD^{-1} A) \Delta z = r_s - A^TD^{-1}r_c,\\
D\Delta \nu = r_c + CA\Delta z.
\end{gather}
\end{subequations}
Under (A\ref{assump:ker_intersection}) we have that $H + A^T CD^{-1} A \succ 0$ and thus we are able to reduce a general $q+n \times q+n$ system to a $n \times n$ dense symmetric positive definite system, which can be solved using a Cholesky factorization or the conjugate gradient method, and a diagonal $q \times q$ system. 

% Note that unless structure present in $C,D,$ and $A$ is exploited the computational cost of computing $H + A^T CD^{-1} A$ has the potential to dominate the factorization.

% \begin{rmk}
% If $v_i^2 + y_i^2 + \eps^2 \approx 0$ and $(y_i,v_i) > 0$ then $V$ can become ill-conditioned. While it is possible to handle this ill-conditioning using smoothing alone, due to the form in which $\eps$ enters into \eqref{eq:CD_smooth}, a large $\eps$ is required to achieve reasonable conditioning in this manner. Since this may unnecessarily impact convergence it motivates the addition of the regularization term as a practical way to improve conditioning without excessive smoothing.
% \end{rmk}

\section{Numerical experiments}
In this section we compare FBRS against other methods in terms of execution time on both real-time hardware and on a regular computer. Execution times on a computer are of relevance during initial controller development and allow us to compare FBRS against a wider array of solvers. FBRS was compared against several state of the art solvers: (i) \texttt{quadprog} (MATLAB 2015a SP1) Interior Point (IP), (ii) QPKWIK (dual active set) \cite{schmid1994quadratic}, (iii) ECOS (self-dual interior point) \cite{domahidi2013ecos}, (iv) GPAD (dual accelerated gradient projection) \cite{patrinos2014accelerated}, and (v) PDIP (primal-dual interior point)\cite[Algorithm 14.3]{nocedal2006numerical}. We use the norm of the natural residual,
\begin{equation}
	F_{NR}(x) = \begin{bmatrix}
	\nabla_z L(z,v)^T &
	\mathrm{min}(y,v)^T
	\end{bmatrix}^T,
\end{equation}
in this section to quantify the quality of a solution.

\subsection{Comparisons against other methods} \label{ss:comp_pc}
Four example problems of varying sizes were considered: A convex MPC controller for asteroid circumnavigation \cite{liao2016model}, a diesel engine Economic MPC problem \cite{liaomcpherson2017diesel}, an extended command governor for an F16 control problem, similar to the reference governor presented in \cite{kalabic2013prioritization}, and an MPC controller for a spacecraft attitude control problem, similar to the example presented in \cite{butts2016perturbed}. Each control problem generates a sequence of related, feasible, QPs. The solution of the previous QP is used to initialize the next for all algorithms that accept a warmstart. Table~\ref{tab:problem_sizes} summarizes the size of each problem and the number of QPs in the sequence\footnote{For all experiments performed on a laptop each QP in each sequence is solved 25 times and the measured execution time is averaged to attenuate variability in execution time caused by the operating system (OS).}.

The results\footnote{All experiments were performed on a 2015 i7 Macbook Pro with 16GB of memory running MATLAB 2015a SP1. FBRS, QPKWIK, PDIP, and GPAD were implemented in the MATLAB language and compiled into mex functions using MATLAB Coder. We implemented matrix factorization updating for QPKWIK. All second order methods were limited to 30 major iterations except for QPKWIK which used up to 500; GPAD was allowed up to 3000 iterations. FBRS, and PDIP were terminated when $||F_{NR}|| \leq 10^{-4}$. The solution tolerances for QUADPROG, GPAD and ECOS were tuned until the solution errors, using the same metric, were of the same order of magnitude as the other methods. QPKWIK does not have an adjustable error tolerance, $10^{-4}$ was used as the constraint tolerance.} are shown in Tables~\ref{tab:tave} and \ref{tab:tmax} which summarize the average and maximum execution times for each QP sequence, respectively. FBRS performed well on the smaller problems, with performance similar to the dual active set method on the diesel and spacecraft examples. FBRS was soundly beaten on the F16 problem which has many constraints and few variables. Two interior point methods (PDIP and QUADPROG IP) had the best worst case execution times on the asteroid example (the largest considered) by a small margin. FBRS had the best average execution time and was competitive in terms of worst case execution time. Overall, FBRS, despite its simplicity, was found to be competitive with state of the art solvers for both large and small scale problems.

\begin{table} 
\centering
\caption{Summary of problem sizes.}
\label{tab:problem_sizes}
\begin{tabular}{|c|c|c|c|c|} \hline
 & Asteroid & Diesel & F16 & S/C \\\hline
 Number of variables& $280$ & $30$ & $12$ & $31$ \\\hline
 Number of constraints & $490$ & $70$ & $1010$ & $181$ \\ \hline
 Number of QPs in sequence & $300$ & $3000$ & $599$ & $50$ \\\hline
\end{tabular}
\end{table}

\begin{table} 
\centering
\caption{Average execution time for each sequence of QPs. Entries in each column have been normalized by the first element.}
\label{tab:tave}
\begin{tabular}{|c|c|c|c|c|}
\hline 
 & Asteroid & Diesel & F16 & S/C\\ \hline
 Normalization [msec] & $17.75$ & $0.11$ & $2.14$ & $0.18$ \\ \hline \hline
 FBRS & $1$ & $1$ & $1$ & $1$ \\ \hline
 PDIP & $4.06$ & $9.62$ & $0.45$ & $7.40$ \\ \hline
 QPKWIK & $4.28$ & $1.12$ & $0.08$ & $0.54$ \\ \hline
 GPAD & $N/A$ & $17.94$ & $10.97$ & $0.81$ \\ \hline
 Quadprog IP & $3.60$ & $75.04$ & $16.23$ & $121.77$ \\ \hline
 ECOS & $9.01$ & $23.99$ & $8.54$ & $54.87$ \\ \hline
\end{tabular}

\end{table}

\begin{table}
\centering
\caption{Maximum execution time for each sequence of QPs. Entries in each column have been normalized by the first element.}
\label{tab:tmax}
\begin{tabular}{|c|c|c|c|c|}
\hline 
 & Asteroid & Diesel & F16 & S/C\\ \hline
 Normalization [msec] & $164.67$ & $1.16$ & $11.40$ & $0.26$ \\ \hline \hline
 FBRS & $1$ & $1$ & $1$ & $1$ \\ \hline
 PDIP & $0.94$ & $2.26$ & $0.18$ & $5.56$ \\ \hline
 QPKWIK & $3.24$ & $0.73$ & $0.06$ & $0.91$ \\ \hline
 GPAD & $N/A$ & $63.13$ & $8.80$ & $1.06$ \\ \hline
 Quadprog IP & $0.96$ & $8.22$ & $5.77$ & $112.42$ \\ \hline
 ECOS & $1.31$ & $3.95$ & $2.40$ & $44.82$ \\ \hline
\end{tabular}

\end{table}

\subsection{Comparisons on embedded hardware}
The performance of the FBRS solver was compared against other methods on embedded hardware. Specifically the economic MPC controller from \cite{liaomcpherson2017diesel} was placed in closed loop with a high fidelity model of a diesel engine. The model and controller, including the QP solvers which were implemented using embedded MATLAB, were implemented in Simulink (2010b SP2) and loaded onto a DS1006 rapid prototyping unit using real-time workshop (2.8 GHz CPU, 1 GB RAM). We used $||F_{NR}|| \leq 10^{-4}$ as a stopping criterion and measured the turnaround time of the QP solvers using the profiling tools supplied with the processor board. ECOS and \texttt{quadprog} could not be loaded onto the DS1006 board since they are not compatible with the real-time workshop build process. FBRSacc and PDIPacc disable all safeguards and use structured linear algebra to speed up computation of matrix operations, e.g., \eqref{eq:Kred}.
\begin{rmk}
Note that the DS1006 runs a real-time OS, and thus the turnaround time is the precise metric used to judge if an application is executable in real-time.
\end{rmk}
Figures~\ref{fig:HIL_cold} and \ref{fig:HIL_warm} illustrate the results when the solver are cold-started and warm-started respectively. GPAD was unable to converge when cold-started and QPKWIK performed very poorly, this is surprising given that it performed well during the tests in Section~\ref{ss:comp_pc}\footnote{We attribute this to subtleties in the automatic codegeneration process which we are still investigating. The code used to implement the QPKWIK algorithm on the DS1006 board is identical to the code used during the numerical trials detailed in Section~\ref{ss:comp_pc}.}.

A summary of the results are given in Table~\ref{tab:HIL}, FBRSacc performed best on average and was able to efficiently exploit warm-starting; PDIPacc was the most efficient cold-start algorithm. FBRSacc and PDIPacc had effectively (within 5 microseconds) identical worst case turnaround times.  Note that the sampling period of the diesel airpath control application is 8 msec; both FBRSacc and PDIPacc are thus real-time executable.

\begin{figure}
	\centering
	\includegraphics[width=0.40\textwidth]{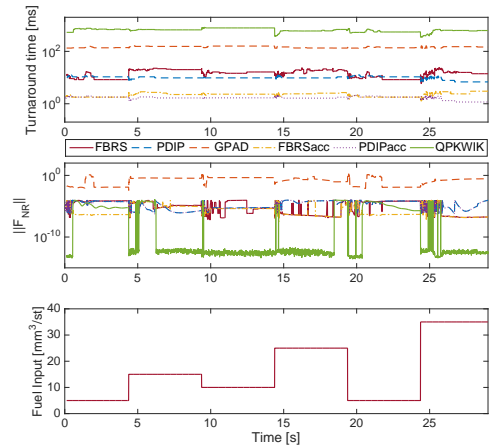}
	\caption{Measured turnaround time, residual, and the excitation sequence for the embedded comparisons. All methods were cold started.}
	\label{fig:HIL_cold}
\end{figure}

\begin{figure}
	\centering
	\includegraphics[width=0.40\textwidth]{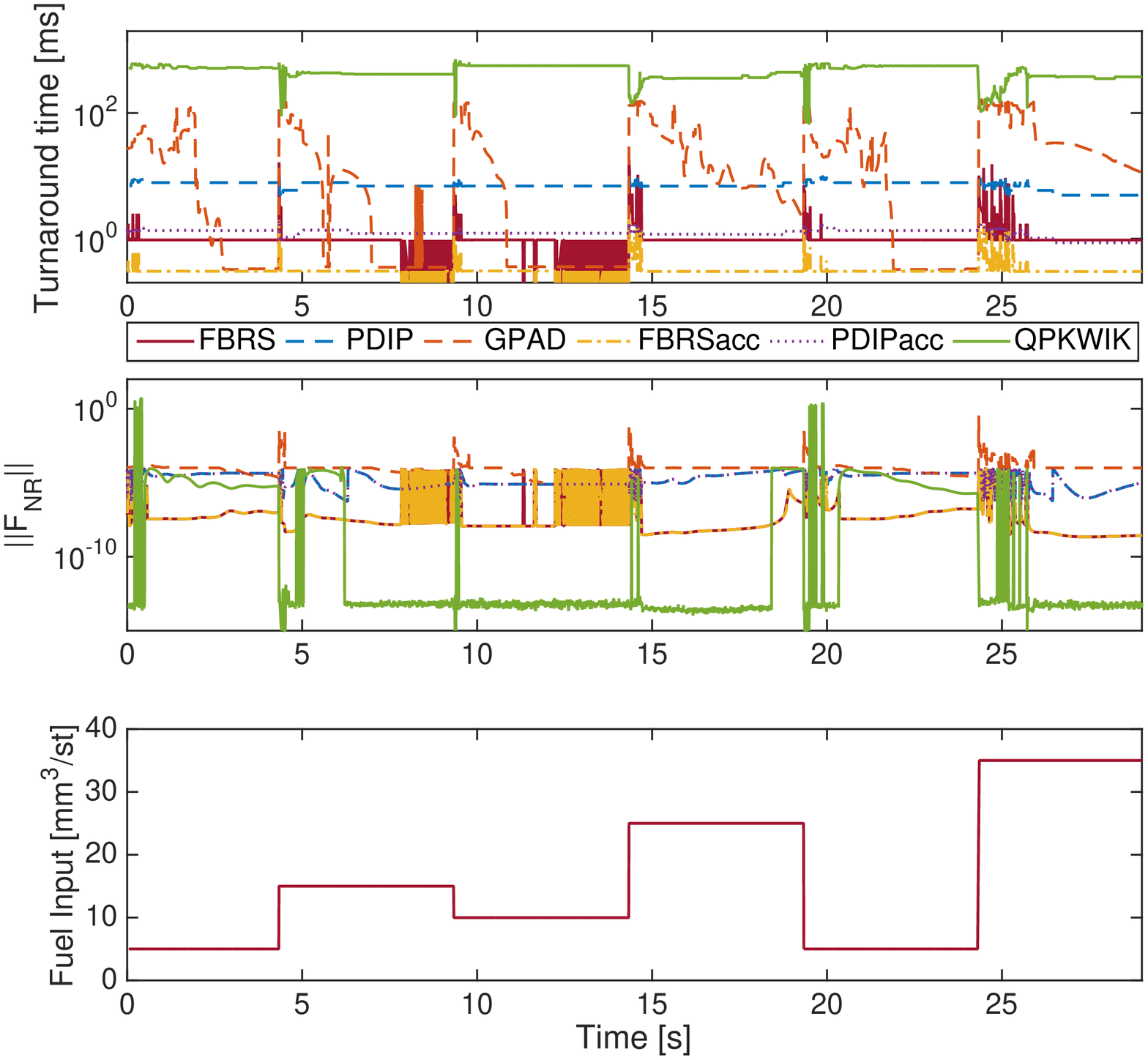}
	\caption{Measured turnaround time, residual, and the excitation sequence for the embedded comparisons. All methods were warm started.}
	\label{fig:HIL_warm}
\end{figure}

\begin{table}
\centering
\caption{Summary of the embedded testing. Times are in msec. ERR indicates the maximum value of $||F_{NR}||.$}
\label{tab:HIL}
\begin{tabular}{|c|c|c|c|c|c|c|}
\hline 
& \multicolumn{3}{|c|}{WARM} &  \multicolumn{3}{|c|}{COLD} \\ \hline
     & AVE & MAX & ERR & AVE & MAX & ERR\\ \hline
FBRS & $1.02$ & $17.68$ & $10^{-4}$ & $14.8$ & $23.8$ & $10^{-4}$\\ \hline
PDIP & $7.04$ & $11.65$ &$10^{-4}$& $9.84$ & $12.7$ & $10^{-4}$ \\\hline
GPAD & $28.6$ & $162.1$ & $0.35$ & $146$ & $163$ & $1.8$\\ \hline
FBRSacc & $0.31$ & $2.04$ & $10^{-4}$ & $2.25$ & $3.02$ & $10^{-4}$\\ \hline
PDIPacc & $1.23$ & $1.99$ & $10^{-4}$ & $1.67$ & $2.27$ & $10^{-4}$\\ \hline
QPKWIK & $465$ & $712$ & $5.4$ & $649$ & $841$ & $1.2\cdot10^{-4}$\\ \hline
\end{tabular}

\end{table}

\section{Convergence of the algorithm} \label{ss:convergence}
In this section we analyze the convergence properties of FBRS. We consider the semismooth variant of FBRS (see Remark~\ref{rmk:FBRSss}) which allows $\eps \geq0$. Several properties of the mapping $F_\eps$ and the merit function are established and then local and global convergence results are obtained.

\subsection{Key properties of the mapping} \label{ss:F_properties}
Here we establish some properties which will be needed to prove the convergence of FBRS. These properties hold for all $\eps \geq0$; in this section we focus on the case where $\eps = 0$, if $\eps > 0$ then strong semismoothness is implied by continuous differentiability and CD regularity is implied by Jacobian non-singularity.
\begin{prp}
The mapping $F_\eps: \reals^{n+q} \times \reals_{\geq0}\mapsto \reals^{n+q}$ has the following properties.
\begin{enumerate}
  \item $F_\eps$ is locally Lipschitz continuous i.e., for every $x\in \reals^{n+q}$ there exists $L_F(x) >0 $ and a neighbourhood $O$ of $x$ such that
\begin{equation} \label{eq:F_lip}
||F_\eps(x+\xi) - F_\eps(x)|| \leq L_F ||\xi||~~ \forall \xi \in O.
\end{equation}
\item The mapping $F_\eps$ is strongly semismooth.
\item The mapping $F_\eps$ is CD (Clarke Differential) regular \cite{qi1997semismooth} in the vicinity of a root $\bar{x}$ which satisfies $F_\eps(\bar{x}) = 0$. This implies that there exists $L_I$ and a neighbourhood $S$ of $\bar{x}$ such that
\begin{equation} \label{eq:CDreg}
  ||x - \bar{x}|| \leq L_I ||F(x)||~~ \forall x \in S.
\end{equation}

\item Define the error matrix $E = V-K$ as the difference between the regularized and unregularized iteration matrices and let $\delta$ denote the regularization parameter. Then
\begin{equation} \label{eq:E_bound}
   \exists~\gamma > 0 \text{ such that } ||E|| \leq \gamma \delta,~ \forall x \in \reals^{n+q}.
\end{equation}

\end{enumerate}
\end{prp}
\begin{proof}
Result 1: This follows from the Lipschitz continuity of affine functions and of the FB function \cite{fischer1992special}.

Result 2: The mapping $F_\eps$ is the concatenation of an affine function and the composition of an affine function and the Fischer-Burmeister transform. Affine functions are strongly semismooth as is the FB function \cite[Lemma 16]{fischer1997solution}. Further, the concatenation and composition of (strong) semismooth mappings are (strongly) semismooth, see, e.g., \cite[Propositions 1.73 and 1.74]{izmailov2014newton}. 

Result 3: CD regularity can be established by noting that the CD regularity of the Fisher-Burmeister mapping applied to the KKT conditions of a general nonlinear program, assuming the SSOSC and the LICQ, was proven in \cite[Lemma 4.2]{fischer1992special}. The CD regularity of the convex QPs considered in this paper then follows as a special case. 

%Note that if $\eps >0$ then CD regularity reduces to Jacobian non-singularity which holds by Theorem~\ref{thm:V_nonsingularity}.

Result 4: This bound can be directly computed by inspecting \eqref{eq:newton_system} and applying the properties of norms with $\gamma = 1 + ||A||$.
\end{proof}

\subsection{Key properties of the merit function} \label{ss:merit_properties}
\begin{prp}
The merit function $\theta_\eps = \frac12||F_\eps||_2^2:\reals^{n+q} \times \reals_{\geq0} \mapsto \reals_{\geq0}$ has the following properties.
\begin{enumerate}
  \item For any $\eps \geq 0$ the merit function $\theta_\eps(x)$ is continuously differentiable.
  \item The gradient of the merit function can be computed as
  \begin{equation}
    \nabla_x \theta_\eps(x) = V^TF_\eps(x),
  \end{equation}
  for any $V \in \partial_C F_\eps(x)$. 
  \item For any $\eps \geq 0$ the merit function has a unique minimizer $\bar{x}$ which corresponds to $F_\eps(\bar{x}) = 0$.
\end{enumerate}
\end{prp}

\begin{proof}

Result 1: If $\eps >0$ then $F_\eps$ is continuously differentiable so we consider the case when $\eps = 0$. Following \cite{facchinei1997new}, consider $\partial \theta^T = \partial F^T F \subseteq \partial_C F^T F$, which holds by the calculus of the Generalized Jacobian (\cite{clarke1990optimization}, Theorems 2.6.6, 2.2.4). Expanding the product we obtain
\begin{equation} \label{eq:grad_theta}
  \partial \theta^T = \begin{bmatrix}
    H \nabla_zL + A^T C \phi\\
    A\nabla_z L + D\phi
  \end{bmatrix},
\end{equation}
the product $w = C\phi$, when written elementwise, is of the form $w_i = \gamma_i \phi(y_i,v_i)$. Since $\gamma_i$ is multivalued only if $(v_i,y_i) = 0$, (see Remark~\ref{rmk:K_compute}) which in turn implies that $\phi(v_i,y_i) = 0$, the multivalued elements of $w_i = \gamma_i \phi(y_i,v_i)$ are ``zeroed out''. The same argument holds for $D\phi$ and thus the products $C\phi$ and $D\phi$ must be single valued implying that $\partial \theta^T = \nabla_x \theta$ is continuously differentiable \cite[Corollary to Theorem 2.2.4]{clarke1990optimization}. 

Result 2: Since $\partial \theta = \partial_C \theta \subseteq F^T \partial_C F$ and $F^T V$ is a singleton for any $V \in \partial_C F$, we must have that $\nabla_x \theta = \partial \theta^T = V^T F$ for any $V \in \partial_C F$.

Result 3: The necessary conditions for minimizing the merit function are $\nabla_x \theta_\eps = V^T F_\eps(x) = 0$ for any $V \in \partial_C F_\eps(x)$. If $\eps > 0$ then $V$ is always nonsingular (Theorem~\ref{thm:V_nonsingularity}) and $V^T F_\eps = 0$ if and only if $F_\eps(x) = 0$. If $\eps = 0$ then \cite[Theorem 4.4]{facchinei1998regularity} can be invoked in place of Theorem~\ref{thm:V_nonsingularity} to obtain the same result.
\end{proof}

\subsection{Asymptotic convergence}
This section focuses on local convergence of FBRS for some fixed $\eps \geq 0$. We will use $k \in \mathbb{Z}_+$ as a superscript as the iteration counter for the algorithm. Let $x^*$ denote the root of $F_\eps$ which exists and is unique under our assumptions so that $F_\eps^* = F_\eps(x^*) = 0$. We define the error $e^k = x^k -x^*$, and the matrices $V \in \partial_C F_\eps(x^k,\eps)$ and $K = V + \nabla_xR(x^k,\delta)$.

\begin{thm} \label{thrm:asym_conv}
Let $\{x^k\}$, $\{\Delta x^k\}$ be generated by FBRS and pick any fixed $\eps \geq 0$. Then there exists $\eta>0$ and a neighbourhood $U$ of the root $x^*$ of $F_\eps$, such that if $x^{0} \in U$ then the bound $||e^{k+1}|| \leq \eta ||e^k||^2$ holds, and $\{x^k\} \in U $ converges quadratically to $x^*$.
\end{thm}

\begin{proof}
\begin{subequations}
Consider the update equation
\begin{equation}
  ||e^{k+1} || = ||x^{k+1} - x^*||  = ||x^k - x^* + \Delta x^k||,
\end{equation}
and recall that $\Delta x^k = -K^{-1} F_\eps$. Combining these two and performing some algebraic manipulations we obtain
\begin{align}
||e^{k+1}|| & \leq ||e^k- K^{-1}F_\eps||\\
& \leq ||K^{-1}||~||Ke^k- F_\eps + F_\eps^*|| \label{eq:factorKinv}\\
& \leq M ||Ke^k- Ve^k+ Ve^k- F_\eps + F_\eps^*|| \label{eq:add_sub}\\
& \leq M ||Ke^k- Ve^k|| + M ||Ve^k- F_\eps + F_\eps^*|| \\
& \leq M ||K-V||~||e^k|| + M ||Ve^k- F_\eps + F_\eps^*||\label{eq:split},
\end{align}
where $M \geq ||K^{-1}(x)||, ~\forall x \in U_3$, where $U_3$ is any compact neighbourhood of $x^*$. The existence of $M$ is guaranteed by the non-singularity of $K$ (Corollary~\ref{corr:K_nonsingular}). The first term in \eqref{eq:split} represents the error induced by regularization, using \eqref{eq:E_bound} we have the following bound
\begin{equation} \label{eq:KV_asym}
  M ||K-V||~||e^k|| \leq \gamma M \delta^k~||e^k||.
\end{equation}
For the second term, following \cite{qi1993nonsmooth} and \cite{chen1998global}, the strong semismoothness of $F_\eps$ implies that there exists a neighbourhood $U_1$ of $x^*$ and a constant $T >0 $ such that 
\begin{align} \label{eq:semi_smooth_asym}
  ||Ve^k - F_\eps + F_\eps^*|| &\leq
   \sqrt{\sum_{i=1}^{n+q} ||V_ie^k- F_{\eps,i} + F^*_{\eps,i}}||^2 \\
   &\leq T ||e^k||^2, \quad \forall x \in U_1
\end{align}
Combining \eqref{eq:split}, \eqref{eq:semi_smooth_asym}, \eqref{eq:KV_asym}, that $\delta^k \leq  ||F_\eps(x)||$ by the construction of FBRS (Step~\ref{step:reg_update}), and the Lipshitz continuity of $F$ \eqref{eq:F_lip} yields
\begin{equation} \label{eq:eta_bar}
  ||e^{k+1}|| \leq M(\gamma L_F^2 + T)||e^k||^2 ~~\forall x \in U,
\end{equation}
where $U = U_1 \cap U_2 \cap U_3$, and $U_2$, $L_F = L_F(x^*)$ are the neighbourhood and Lipschitz constant in \eqref{eq:F_lip}.  Letting $\eta = M(\gamma L_F^2 + T)$ completes the proof.
\end{subequations}
\end{proof}

% \begin{rmk}
% As previously noted in \cite{fischer1992special} no assumption about strict complimentary slackness is required for quadratic convergence.
% \end{rmk}

\subsection{Global convergence}
This section provides a short proof of the global convergence properties of FBRS.

\begin{lmm} \label{lmm:step_acceptance}
Let $\{x^k\}$ and $\{\Delta x^k\}$ be generated by FBRS. Assume that $\theta_\eps(x^k) \neq 0$ and $ 1-\gamma ||K^{-1}|| \delta^k > \sigma$, where $\gamma$ is the constant in \eqref{eq:E_bound}. Then there exists a step length $t^k \in (0,1]$ such that the Armijo condition, 
\begin{equation} \label{eq:armijo}
\theta_\eps(x^k+t^k\Delta x^k) \leq  (1-2t^k\sigma) \theta_\eps(x^k),
\end{equation}
is satisfied.
\end{lmm}

\begin{proof}
Consider a fixed but arbitrary iteration $k$; from this point forward we drop the iteration superscript to steamline the presentation of the proof. As $\theta_\eps$ is continuously differentiable (see section~\ref{ss:merit_properties}) we can invoke the fundamental theorem of calculus to write that
\begin{subequations}
\begin{align}
  \theta_\eps(x + t\Delta x) &= \theta_\eps(x) + t \nabla_x \theta_\eps(x)^T \Delta x \nonumber\\
   &+ t \int_0^1 [\nabla_x \theta(x+t\Delta x \lambda) - \nabla\theta_\eps(x)]^T\Delta x ~ d\lambda,
\end{align} 
defining $\Delta \theta(t) = \theta_\eps(x + t\Delta x) - \theta_\eps(x)$ and using that $\nabla_x \theta_\eps(x) = V^T F(x)$, for any $V \in \partial_C F(x)$ (see section~\ref{ss:merit_properties}) yields that
\begin{align}
\Delta \theta(t) &= t \nabla_x \theta_\eps^T \Delta x + t \int_0^1 [\nabla_x \theta_\eps(x+t\Delta x \lambda) - \nabla\theta_\eps]^T\Delta x ~ d\lambda,\\
=& -t F^T V K^{-1} F + t \int_0^1 [\nabla_x \theta_\eps(x+t\Delta x \lambda) - \nabla\theta_\eps] ^T\Delta x ~ d\lambda.
\end{align}
Substituting in $VK^{-1} = I + EK^{-1}$ yields
\begin{align}
\Delta \theta (t) &\leq-t ||F||^2(1-||E K^{-1}||) \nonumber \\
& + t \int_0^1 [\nabla_x \theta_\eps(x+t\Delta x \lambda) - \nabla\theta_\eps]^T\Delta x~d\lambda,
\end{align}
rearranging the bound on $||E||$ \eqref{eq:E_bound}, letting $M(\delta) = ||K^{-1}||$, and taking norms of the remaining positive terms yields the following estimate

\begin{align}
\Delta \theta(t) &\leq -2t \theta_\eps (1- \gamma M \delta) \nonumber \\
&+ t \int_0^1 ||\nabla_x \theta_\eps(x+t\Delta x \lambda) - \nabla\theta_\eps ||~||\Delta x|| d\lambda.
\end{align}
Since $\nabla_\eps \theta$ is Lipschitz, see \eqref{eq:grad_theta}, letting $L_\theta$ be its Lipschitz constant and integrating we obtain that 
\begin{align}
\Delta \theta(t) &\leq -2t \theta_\eps (1- \gamma M \delta) + t \int_0^1 t L_\theta ~||\Delta x||^2 d\lambda,\\
&\leq -2t \theta_\eps (1- \gamma M \delta) + \frac12 t^2 L_\theta ~||K^{-1}F||^2,\\
&\leq -2t \theta_\eps (1- \gamma M \delta) + t^2 L_\theta ~M^2 \theta_\eps.
\end{align}
\end{subequations}
From the last inequality we can conclude that there exists a sufficiently small $t$ such that the Armijo condition is satisfied, in particular any $t < \hat{t}$ where
\begin{equation}
  \hat{t} = \frac{2(1-\gamma M \delta - \sigma)}{L_\theta M^2},
\end{equation}
will be accepted by the algorithm. Further, FBRS uses a backtracking line search with backtracking factor $\beta \in (0,1)$ so we can conclude that $t \geq \beta \hat{t}$ bounding $t$ away from zero.
% \begin{equation}
%   t \geq \beta \frac{2(1-\gamma M \delta - \sigma)}{L_\theta M^2} > 0,
% \end{equation}
\end{proof}

\begin{corollary} \label{corr:small_delta_exists}
For all $x \in \{x^k\}$ there exists $\bar{\delta}$ such that $\Delta x $ will be a direction of sufficient decrease for $\theta_\eps$ if $\delta < \bar{\delta}$.
\end{corollary}

\begin{proof}
A sufficient condition for $\Delta x$ to be a direction of sufficient descent for $\theta_\eps$ is $1- \gamma M \delta = 1- \gamma \delta ||(V-E(\delta))^{-1}||> \sigma$. Since $V$ is always invertible (Theorem~\ref{thm:V_nonsingularity}) and $K = V-E$ is invertible for any $\delta \geq 0$ (Corollary~\ref{corr:K_nonsingular}) then $||(V-E)^{-1}|| \to ||V^{-1}||$ as $\delta \to 0$ and $\Gamma(\delta) = \gamma \delta ||(V-E)^{-1}|| \to 0$ as $\delta \to 0$. The existence of $\bar{\delta}$ then follows from the continuity of $\Gamma(\delta)$.
\end{proof}

\begin{thm} \label{thm:global_convergence}
Let the assumptions in section~\ref{ss:QP_formulation} and Lemma~\ref{lmm:step_acceptance} hold and let the sequence $\{x^k\}$ be generated by FBRS. Then for all initial points, $x^0 \in \reals^{n+q}$, the sequence $\{x^k\}$ is well defined and $\{x^k\} \rightarrow x^*$ as $k \rightarrow \infty$.
\end{thm}
\begin{proof}
We begin by noting that, by Corollary~\ref{corr:K_nonsingular}, the iteration matrix $K(x,\eps)$ is always non-singular; as a result the sequence $\{x^k\}$ generated by FBRS is unique and well defined for any initial condition.

Let $\Delta x(x^k,\delta^k)$ be generated by FBRS. Consider the merit function $\theta_\eps$; if $\theta_\eps(x) > 0$ then if $\delta$ is chosen sufficiently small, which is always possible by Corollary~\ref{corr:small_delta_exists}, then $\Delta x$ will be a direction of sufficient descent for $\theta_\eps$. Thus invoking Lemma~\ref{lmm:step_acceptance} we have that
\begin{equation}
  \theta_\eps(x^{k+1}) < (1-2t_k\sigma) \theta_\eps(x^k),
\end{equation}
as $t_k\in (0,1]$ and $\sigma \in (0,0.5)$ $\{\theta_\eps(x^k)\}$ is a strictly decreasing sequence. Since $\theta_\eps$ is bounded from below by zero $\{\theta_\eps(x^k)\}$ must converge to some $\theta^* \geq 0$ as $k \to \infty$ and, as $1-2t_k\sigma <1$, we must have that $\theta^* = 0$. Noting that $\theta_\eps(x) = 0$ if and only if $F_\eps(x) = 0$ and that $F_\eps(x) = 0$ if and only if $x = x^*$ completes the proof.
\end{proof}

% \begin{rmk}
% The key property for globalization of the algorithm is that $V$ is always non-singular. This excludes the existence of stationary points of the merit function that do not satisfy $F_\eps(x) = 0$.
% \end{rmk}

\subsection{Acceptance of unit steps}
In this section we prove that once the iterates are sufficiently close to the solution then the linesearch will accept unit steps, allowing FBRS to recover the fast asymptotic convergence rates of Theorem~\ref{thrm:asym_conv}.
\begin{thm} \label{thm:q_recovery}
Let the assumptions in section~\ref{ss:QP_formulation} hold and let $\{x^k\}$ and $\{\Delta x^k\}$ be generated by FBRS. Then there exists a neighbourhood $X$ of the solution $x^*$ such that $\theta_\eps(x^k + \Delta x^k) \leq (1-2\sigma) \theta_\eps(x^k),~ \forall x \in X$, implying that the linesearch will accept unit steps.
\end{thm}

\begin{proof}
\begin{subequations}
Choose a fixed but arbitrary iteration $k$; from this point forward we drop the iteration superscript to steamline the presentation of the proof. Consider,
\begin{align}
\theta_\eps(x+\Delta x) &= \frac12 ||F_\eps(x+\Delta x) - F_\eps(x^*)||^2,
\end{align}
using the Lipshitz continuity of $F_\eps$ and that, by Theorem~\ref{thrm:asym_conv}, there exists $\eta > 0$ such that $||x+\Delta x - x^*|| \leq \eta ||x - x^*||^2$ in some neighbourhood $U$ of $x^*$ we can conclude that
\begin{align}
\theta_\eps(x+\Delta x) &\leq \frac12 L_F^2 ||x + \Delta x - x^*||^2,\label{eq:lpF1}\\
&\leq \frac12 L_F^2 \eta ||x - x^*||^4,\label{eq:lc1}
\end{align}
for all $x \in U$. The CD regularity of $F_\eps$, see section~\ref{ss:F_properties}, implies that there exists a neighbourhood $S$ of $x^*$ and $L_I >0$ such that $||x - x^*|| \leq L_I ||F_\eps(x)||~~ \forall x \in S$, thus
\begin{align}
  \theta_\eps(x+\Delta x) &\leq \frac12 L_F^2 \eta L_I^2 ||x-x^*||^2 ||F_\eps(x)||^2,\\
  &\leq L_F^2 \eta L_I^2 ||x-x^*||^2 \theta_\eps(x), 
\end{align}
for all $x\in U\cap S$. By continuity of $||x - x^*||$ there then must exist $\hat{x}$ such that 
\begin{equation}
  L_F^2 \eta L_I^2 ||\hat{x} - x^*||^2 = 1-2\sigma,
\end{equation}
and thus we have that
\begin{equation}
  \theta_\eps(x + \Delta x) \leq (1-2\sigma) \theta_\eps(x),
\end{equation}
for all $x$ such that $||x-x^*|| \leq ||\hat{x} - x^*||$. Setting $X = \{x \in U \cap S~|~||x-x^*|| \leq ||\hat{x} - x^*|| \}$ completes the proof.
\end{subequations}
\end{proof}

\section{Conclusion}
This paper presented a regularized and smoothed Fischer-Burmeister method for solving convex QPs. The method is attractive for real-time and embedded applications since its simple to code, easy to warmstart, and its performance is competitive with other state of the art solvers. Future work includes extending the method to more general convex problems e.g., SOCPs, and considering problems with non-unique dual solutions.

% use section* for acknowledgment
\section*{Acknowledgment}
The authors would like to thank Shinhoon Kim, Marco Nicotra, and Ken Butts.

% Can use something like this to put references on a page
% by themselves when using endfloat and the captionsoff option.
\ifCLASSOPTIONcaptionsoff
  \newpage
\fi

\bibliography{fbsmooth.bib}

\begin{thebibliography}{10}

\bibitem{boyd2004convex}
S.~Boyd and L.~Vandenberghe, {\em Convex optimization}.
\newblock Cambridge university press, 2004.

\bibitem{grune2011nonlinear}
L.~Gr{\"u}ne and J.~Pannek, ``Nonlinear model predictive control,'' in {\em
  Nonlinear Model Predictive Control}, pp.~43--66, Springer, 2011.

\bibitem{rawlings2009model}
J.~B. Rawlings and D.~Q. Mayne, {\em Model predictive control: Theory and
  design}.
\newblock Nob Hill Pub., 2009.

\bibitem{goodwin2006constrained}
G.~C. Goodwin, M.~M. Seron, and J.~A. De~Don{\'a}, {\em Constrained control and
  estimation: an optimisation approach}.
\newblock Springer Science \& Business Media, 2006.

\bibitem{boggs1995sequential}
P.~T. Boggs and J.~W. Tolle, ``Sequential quadratic programming,'' {\em Acta
  numerica}, vol.~4, pp.~1--51, 1995.

\bibitem{diehl2002real}
M.~Diehl, H.~G. Bock, J.~P. Schl{\"o}der, R.~Findeisen, Z.~Nagy, and
  F.~Allg{\"o}wer, ``Real-time optimization and nonlinear model predictive
  control of processes governed by differential-algebraic equations,'' {\em
  Journal of Process Control}, vol.~12, no.~4, pp.~577--585, 2002.

\bibitem{wang2010fast}
Y.~Wang and S.~Boyd, ``Fast model predictive control using online
  optimization,'' {\em IEEE Transactions on Control Systems Technology},
  vol.~18, no.~2, pp.~267--278, 2010.

\bibitem{domahidi2012efficient}
A.~Domahidi, A.~U. Zgraggen, M.~N. Zeilinger, M.~Morari, and C.~N. Jones,
  ``Efficient interior point methods for multistage problems arising in
  receding horizon control,'' in {\em Decision and Control (CDC), 2012 IEEE
  51st Annual Conference on}, pp.~668--674, IEEE, 2012.

\bibitem{ferreau2008online}
H.~J. Ferreau, H.~G. Bock, and M.~Diehl, ``An online active set strategy to
  overcome the limitations of explicit mpc,'' {\em International Journal of
  Robust and Nonlinear Control}, vol.~18, no.~8, pp.~816--830, 2008.

\bibitem{patrinos2014accelerated}
P.~Patrinos and A.~Bemporad, ``An accelerated dual gradient-projection
  algorithm for embedded linear model predictive control,'' {\em IEEE
  Transactions on Automatic Control}, vol.~59, no.~1, pp.~18--33, 2014.

\bibitem{mattingley2012cvxgen}
J.~Mattingley and S.~Boyd, ``Cvxgen: A code generator for embedded convex
  optimization,'' {\em Optimization and Engineering}, vol.~13, no.~1,
  pp.~1--27, 2012.

\bibitem{bemporad2016quadratic}
A.~Bemporad, ``A quadratic programming algorithm based on nonnegative least
  squares with applications to embedded model predictive control,'' {\em IEEE
  Transactions on Automatic Control}, vol.~61, no.~4, pp.~1111--1116, 2016.

\bibitem{li1997new}
W.~Li and J.~Swetits, ``A new algorithm for solving strictly convex quadratic
  programs,'' {\em SIAM Journal on Optimization}, vol.~7, no.~3, pp.~595--619,
  1997.

\bibitem{patrinos2011global}
P.~Patrinos, P.~Sopasakis, and H.~Sarimveis, ``A global piecewise smooth newton
  method for fast large-scale model predictive control,'' {\em Automatica},
  vol.~47, no.~9, pp.~2016--2022, 2011.

\bibitem{brand2011parallel}
M.~Brand, V.~Shilpiekandula, C.~Yao, S.~A. Bortoff, T.~Nishiyama, S.~Yoshikawa,
  and T.~Iwasaki, ``A parallel quadratic programming algorithm for model
  predictive control,'' {\em IFAC Proceedings Volumes}, vol.~44, no.~1,
  pp.~1031--1039, 2011.

\bibitem{domahidi2013ecos}
A.~Domahidi, E.~Chu, and S.~Boyd, ``Ecos: An socp solver for embedded
  systems,'' in {\em Control Conference (ECC), 2013 European}, pp.~3071--3076,
  IEEE, 2013.

\bibitem{dueri2014automated}
D.~Dueri, J.~Zhang, and B.~A{\c{c}}ikmese, ``Automated custom code generation
  for embedded, real-time second order cone programming,'' {\em IFAC
  Proceedings Volumes}, vol.~47, no.~3, pp.~1605--1612, 2014.

\bibitem{facchinei1998regularity}
F.~Facchinei, A.~Fischer, and C.~Kanzow, ``Regularity properties of a
  semismooth reformulation of variational inequalities,'' {\em SIAM Journal on
  Optimization}, vol.~8, no.~3, pp.~850--869, 1998.

\bibitem{huang2004sub}
Z.-H. Huang, L.~Qi, and D.~Sun, ``Sub-quadratic convergence of a smoothing
  newton algorithm for the p0--and monotone lcp,'' {\em Mathematical
  programming}, vol.~99, no.~3, pp.~423--441, 2004.

\bibitem{fischer1992special}
A.~Fischer, ``A special newton-type optimization method,'' {\em Optimization},
  vol.~24, no.~3-4, pp.~269--284, 1992.

\bibitem{chen2000penalized}
B.~Chen, X.~Chen, and C.~Kanzow, ``A penalized fischer-burmeister
  ncp-function,'' {\em Mathematical Programming}, vol.~88, no.~1, pp.~211--216,
  2000.

\bibitem{huang2015nonlinear}
M.~Huang, H.~Nakada, K.~Butts, and I.~Kolmanovsky, ``Nonlinear model predictive
  control of a diesel engine air path: A comparison of constraint handling and
  computational strategies,'' {\em IFAC-PapersOnLine}, vol.~48, no.~23,
  pp.~372--379, 2015.

\bibitem{qi1993nonsmooth}
L.~Qi and J.~Sun, ``A nonsmooth version of newton's method,'' {\em Mathematical
  programming}, vol.~58, no.~1, pp.~353--367, 1993.

\bibitem{qi1997semismooth}
L.~Qi and H.~Jiang, ``Semismooth karush-kuhn-tucker equations and convergence
  analysis of newton and quasi-newton methods for solving these equations,''
  {\em Mathematics of Operations Research}, vol.~22, no.~2, pp.~301--325, 1997.

\bibitem{martinez1995inexact}
J.~Mart{\'\i}nez and L.~Qi, ``Inexact newton methods for solving nonsmooth
  equations,'' {\em Journal of Computational and Applied Mathematics}, vol.~60,
  no.~1-2, pp.~127--145, 1995.

\bibitem{qi1999survey}
L.~Qi and D.~Sun, ``A survey of some nonsmooth equations and smoothing newton
  methods,'' in {\em Progress in optimization}, pp.~121--146, Springer, 1999.

\bibitem{izmailov2014newton}
A.~F. Izmailov and M.~V. Solodov, {\em Newton-type methods for optimization and
  variational problems}.
\newblock Springer, 2014.

\bibitem{rademacher1919partielle}
H.~Rademacher, ``{\"U}ber partielle und totale differenzierbarkeit von
  funktionen mehrerer variabeln und {\"u}ber die transformation der
  doppelintegrale,'' {\em Mathematische Annalen}, vol.~79, no.~4, pp.~340--359,
  1919.

\bibitem{clarke1990optimization}
F.~H. Clarke, {\em Optimization and nonsmooth analysis}.
\newblock SIAM, 1990.

\bibitem{chen1998global}
X.~Chen, L.~Qi, and D.~Sun, ``Global and superlinear convergence of the
  smoothing newton method and its application to general box constrained
  variational inequalities,'' {\em Mathematics of Computation of the American
  Mathematical Society}, vol.~67, no.~222, pp.~519--540, 1998.

\bibitem{qi1996c}
L.~Qi, ``C-differentiability, c-differential operators and generalized newton
  methods,'' {\em Applied Mathematics Report AMR96/5, University of New South
  Wales, Sydney, Australia}, 1996.

\bibitem{sun1999ncp}
D.~Sun and L.~Qi, ``On ncp-functions,'' {\em Computational Optimization and
  Applications}, vol.~13, no.~1-3, pp.~201--220, 1999.

\bibitem{nocedal2006numerical}
J.~Nocedal and S.~Wright, {\em Numerical optimization}.
\newblock Springer Science \& Business Media, 2006.

\bibitem{jiang1997smoothed}
H.~Jiang, ``Smoothed fischer-burmeister equation methods for the
  complementarity problem,'' {\em Report, Department of Mathematics, University
  of Melbourne, Parkville, Australia}, 1997.

\bibitem{schmid1994quadratic}
C.~Schmid and L.~T. Biegler, ``Quadratic programming methods for reduced
  hessian sqp,'' {\em Computers \& chemical engineering}, vol.~18, no.~9,
  pp.~817--832, 1994.

\bibitem{liao2016model}
D.~Liao-McPherson, W.~Dunham, and I.~Kolmanovsky, ``Model predictive control
  strategies for constrained soft landing on an asteroid,'' in {\em AIAA/AAS
  Astrodynamics Specialist Conference}, p.~5507, 2016.

\bibitem{liaomcpherson2017diesel}
D.~Liao-McPherson, S.~Kim, K.~Butts, and I.~Kolmanovsky, ``A cascaded economic
  model predictive control strategy for a diesel engine using a non-uniform
  prediction horizon discretization,'' in {\em 2017 IEEE Conference on Control
  Technology and Applications (CCTA)}, pp.~979--986, Aug 2017.

\bibitem{kalabic2013prioritization}
U.~Kalabic, Y.~Chitalia, J.~Buckland, and I.~Kolmanovsky, ``Prioritization
  schemes for reference and command governors,'' in {\em Control Conference
  (ECC), 2013 European}, pp.~2734--2739, IEEE, 2013.

\bibitem{butts2016perturbed}
K.~Butts, A.~Dontchev, M.~Huang, and I.~Kolmanovsky, ``A perturbed chord
  (newton-kantorovich) method for constrained nonlinear model predictive
  control,'' {\em IFAC-PapersOnLine}, vol.~49, no.~18, pp.~253--258, 2016.

\bibitem{fischer1997solution}
A.~Fischer, ``Solution of monotone complementarity problems with locally
  lipschitzian functions,'' {\em Mathematical Programming}, vol.~76, no.~3,
  pp.~513--532, 1997.

\bibitem{facchinei1997new}
F.~Facchinei and J.~Soares, ``A new merit function for nonlinear
  complementarity problems and a related algorithm,'' {\em SIAM Journal on
  Optimization}, vol.~7, no.~1, pp.~225--247, 1997.

\end{thebibliography}

% biography section
% 
% If you have an EPS/PDF photo (graphicx package needed) extra braces are
% needed around the contents of the optional argument to biography to prevent
% the LaTeX parser from getting confused when it sees the complicated
% \includegraphics command within an optional argument. (You could create
% your own custom macro containing the \includegraphics command to make things
% simpler here.)
%\begin{IEEEbiography}[{\includegraphics[width=1in,height=1.25in,clip,keepaspectratio]{mshell}}]{Michael Shell}
% or if you just want to reserve a space for a photo:

% \begin{IEEEbiography}{Michael Shell}
% Biography text here.
% \end{IEEEbiography}

% % if you will not have a photo at all:
% \begin{IEEEbiographynophoto}{John Doe}
% Biography text here.
% \end{IEEEbiographynophoto}

% % insert where needed to balance the two columns on the last page with
% % biographies
% %\newpage

% \begin{IEEEbiographynophoto}{Jane Doe}
% Biography text here.
% \end{IEEEbiographynophoto}

% You can push biographies down or up by placing
% a \vfill before or after them. The appropriate
% use of \vfill depends on what kind of text is
% on the last page and whether or not the columns
% are being equalized.

%\vfill

% Can be used to pull up biographies so that the bottom of the last one
% is flush with the other column.
%\enlargethispage{-5in}

% that's all folks
\end{document}